\documentclass[a4paper, 11pt]{article}

\usepackage{longtable}
\usepackage{geometry}
\usepackage{graphicx}
\usepackage{subfig}
\usepackage{color}
\usepackage{xcolor}
\usepackage{bm}
\usepackage{fixmath}
\usepackage{mathrsfs}
\usepackage{ulem}
\usepackage{multirow}
\usepackage{pdflscape}
\usepackage{rotating}
\usepackage{lscape}
\usepackage{pdflscape}
\usepackage{titlesec}
\usepackage{float}
\usepackage{hyperref}
\usepackage{tabularx}
\usepackage{mathabx}
\usepackage{wasysym}
\usepackage[utf8]{inputenc}
\usepackage[english]{babel}
\usepackage{amsthm}

\newtheorem{theorem}{Theorem}

\newtheorem{lemma}{Lemma}

\titleclass{\subsubsubsection}{straight}[\subsection]

\newcounter{subsubsubsection}[subsubsection]
\renewcommand\thesubsubsubsection{\thesubsubsection.\arabic{subsubsubsection}}

\titleformat{\subsubsubsection}
  {\normalfont\normalsize\bfseries}{\thesubsubsubsection}{1em}{}
\titlespacing*{\subsubsubsection}
{0pt}{3.25ex plus 1ex minus .2ex}{1.5ex plus .2ex}

\makeatletter

\renewcommand\paragraph{\@startsection{paragraph}{5}{\z@}%
  {3.25ex \@plus1ex \@minus.2ex}%
  {-1em}%
  {\normalfont\normalsize\bfseries}}

\renewcommand\subparagraph{\@startsection{subparagraph}{6}{\parindent}%
  {3.25ex \@plus1ex \@minus .2ex}%
  {-1em}%
  {\normalfont\normalsize\bfseries}}

\def\toclevel@subsubsubsection{4}
\def\toclevel@paragraph{5}
\def\toclevel@paragraph{6}
\def\l@subsubsubsection{\@dottedtocline{4}{7em}{4em}}
\def\l@paragraph{\@dottedtocline{5}{10em}{5em}}
\def\l@subparagraph{\@dottedtocline{6}{14em}{6em}}

\makeatother

\def\b{\begin{eqnarray}}
\def\e{\end{eqnarray}}
\def\n{\noindent}

\makeatletter
\newcommand*\bigdot{\mathpalette\bigdot@{.5}}
\newcommand*\bigdot@[2]{\mathbin{\vcenter{\hbox{\scalebox{#2}{$\m@th#1\bullet$}}}}}
\makeatother

\begin{document}

\begin{center}
{\huge \textbf{The Siebeck--Marden--Northshield \vskip.2cm Theorem and the Real Roots of the \vskip.5cm Symbolic Cubic Equation}}

\vspace{9mm}
\noindent
{\large \bf Emil M. Prodanov} \vskip.4cm
{\it School of Mathematical Sciences, Technological University Dublin,
\vskip.1cm
Park House, Grangegorman, 191 North Circular Road, Dublin
D07 EWV4, Ireland,}
\vskip.1cm
{\it e-mail: emil.prodanov@tudublin.ie} \\
\vskip.5cm
\end{center}

\vskip2cm

\begin{abstract}
\n
The isolation intervals of the real roots of the symbolic monic cubic polynomial $x^3 + a x^2 + b x + c$ are determined, in terms of the coefficients of the polynomial, by solving the Siebeck--Marden--Northshield triangle --- the equilateral triangle that projects onto the three real roots of the cubic polynomial and whose inscribed circle projects onto an interval with endpoints equal to stationary points of the polynomial.
\end{abstract}


\newpage

\section{Introduction}
\n
The elegant theorem of Siebeck and Marden (often referred to as Marden's theorem) \cite{m1}--\cite{bad} relates geometrically the complex non-collinear roots $z_1, \,\, z_2,$ and $z_3$ of a cubic polynomial with complex coefficients to a triangle in the complex plane whose vertices are $z_1, \,\, z_2,$ and $z_3$, on one hand, and, on the other, the critical points of the polynomial to the foci of the inellipse of this triangle. This ellipse is unique and is called Steiner inellipse \cite{st}. It is inscribed in the triangle in such way that it is tangent to the sides of the triangle at their midpoints. \\
The real version of the Siebeck--Marden Theorem, as given by Northshield \cite{north}, states that the three real roots (not all of which are equal)
of a cubic polynomial are projections of the vertices of some equilateral triangle in the plane. However, it is the inscribed circle of the equilateral triangle that projects onto an interval the endpoints of which are the stationary points of the polynomial. \\
The goal of this work is to consider a cubic equation with real coefficients and, using the Siebeck--Marden--Northshield theorem \cite{north}, solve the equilateral triangle and find the isolation intervals of the real roots of the symbolic monic cubic polynomial $x^3 + a x^2 + b x + c$.

\section{Analysis}
\n
{\bf Construction}: Any three real numbers, not all equal, are the projections of the vertices of some equilateral triangle in the plane. For the monic cubic polynomial $p(x) = x^3 + a x^2 + b x + c$ with three real roots $x_1, \,\, x_2,$ and $x_3$, not all equal, the vertices of the equilateral triangle --- points $P$, $Q$, and $R$ on Figure 1 --- with coordinates $(x_1, (x_2 - x_3)/\sqrt{3})$, $(x_2, (x_3 - x_1)/\sqrt{3})$, and $(x_3, (x_1 - x_2)/\sqrt{3})$, respectively, project on the roots \cite{north}.  This is the {\it Siebeck--Marden--Northshield triangle}. The inscribed circle of this triangle projects to an interval with endpoints equal to the critical points $\mu_{1,2} = -a/3 \pm (1/3) \sqrt{a^2 - 3b}$ of the cubic polynomial --- the roots of the derivative $p'(x) = 3x^2 + 2ax + b \,$ of $\, p(x)$ \cite{north}. The centroid of the triangle is at $\phi = -a/3$ on the abscissa --- the first coordinate projection of the inflection point of $p(x)$ --- the root of the second derivative $p''(x) = 6x + 2a$. Each side of the triangle is equal to $\alpha = (\sqrt{12}/3) \sqrt{a^2 - 3b}$. The radius of the inscribed circle is $r = (1/3) \sqrt{a^2 - 3b}$. The radius of the circumscribed circle is $2r = (2/3) \sqrt{a^2 - 3b}$.

\begin{figure}[!ht]
\centering
\includegraphics[height=8.8cm, width=0.58\textwidth]{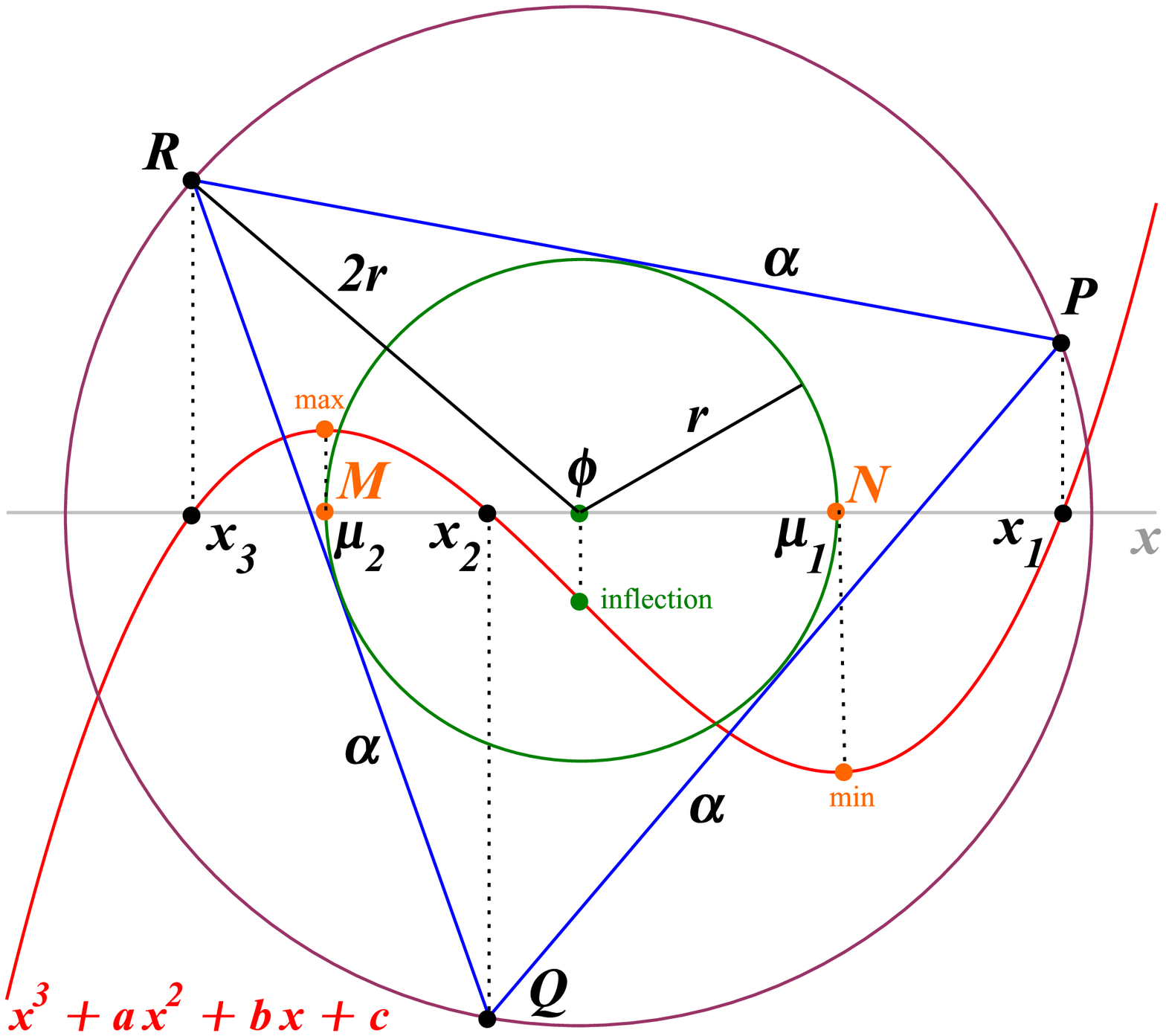}}
{\begin{minipage}{34.5em}
\scriptsize
\vskip.3cm
\begin{center}
{\bf Figure 1}
\end{center}
\vskip-.3cm
{\bf Siebeck--Marden--Northshield Triangle}: When the cubic polynomial $p(x) = x^3 + a x^2 + b x + c$ has three real roots $x_{1,2,3}$ which are not all equal, they can be obtained as projections of the vertices of an equilateral triangle ($PQR$) with coordinates $(x_1, (x_2 - x_3)/\sqrt{3})$, $(x_2, (x_3 - x_1)/\sqrt{3})$, and $(x_3, (x_1 - x_2)/\sqrt{3})$, respectively \cite{north}.
\end{minipage}
\end{figure}

\begin{lemma}
The monic cubic polynomial $p(x) = x^3 + a x^2 + b x + c$ with $b > a^2/3$ has only one real root.
\end{lemma}
\begin{proof}
The discriminant of the monic cubic polynomial $x^3 + a x^2 + b x + c$ is
\b
\Delta_3 = -27 c^2 + (18 a b - 4 a^3) c + a^2 b^2 - 4b^3.
\e
It is quadratic in $c$ and the discriminant of this quadratic is
\b
\Delta_2 = 16 (a^2 - 3b)^3
\e
As $b > a^2/3$, one has $\Delta_2 < 0$ for all $a$ and thus $\Delta_3 < 0$ for all $a$ and $c$. Hence, the cubic polynomial $p(x) = x^3 + a x^2 + b x + c$ with $b > a^2/3$ has only one real root (and a pair of complex conjugate roots).
\end{proof}

\noindent
This can be seen in an easier way: the discriminant of the derivative $p'(x) = 3x^2 + 2ax + b$ is $4(a^2 - 3b)$, hence no critical points of $p(x)$ exist when $b > a^2/3$ and thus $p(x)$ has only one real root. \\
Note that existence of critical points of $p(x)$, warranted by $b \le a^2/3$, does not warrant three real roots. The following Lemma addresses this.

\begin{lemma}
The monic cubic polynomial $p(x) = x^3 + a x^2 + b x + c$ with $b \le a^2/3$ has three real roots, provided that $c \in [c_2, c_1]$, where $c_{1,2}$ are the roots of the quadratic equation
\b
\label{quadratic}
x^2 + \left( \frac{4}{27} \, a^3 - \frac{2}{3}\, a b \right) x  - \frac{1}{27} \, a^2 b^2  + \frac{4}{27}\, b^3 = 0,
\e
namely:
\b
\label{c12}
c_{1,2}(a,b) = c_0 \, \pm \, \frac{2}{27} \, \sqrt{( a^2 - 3 b)^3},
\e
where
\b
\label{c0}
c_0(a,b) = -\frac{2}{27} \, a^3 + \frac{1}{3} \, a b.
\e
\end{lemma}

\begin{proof}
The discriminant $\Delta_3 = -27 c^2 + (18 a b - 4 a^3) c + a^2 b^2 - 4b^3$ of the monic cubic polynomial $x^3 + a x^2 + b x + c$ is positive between the roots of the equation $\Delta_3 = 0$, which is quadratic in $c$. This is exactly equation (\ref{quadratic}) and its roots are the ones given in (\ref{c12}) and (\ref{c0}).
\end{proof}

\begin{lemma}
The maximum distance between the three real roots of the monic cubic polynomial $p(x) = x^3 + a x^2 + b x + c$ is $\sqrt{12} r = (\sqrt{12}/3) \sqrt{a^2 - 3b}$. In this case, one side of the Siebeck--Marden--Northshield triangle is parallel to the abscissa. This is achieved in the case of the ``balanced" cubic --- the one with $c = c_0 = - 2 a^3/27 + a b/3$. For any other $c$ such that $c_2 \le c \le c_1$, the three real roots of the cubic lie within a shorter interval.
\end{lemma}

\begin{figure}[!ht]
\centering
\includegraphics[height=8.8cm, width=0.58\textwidth]{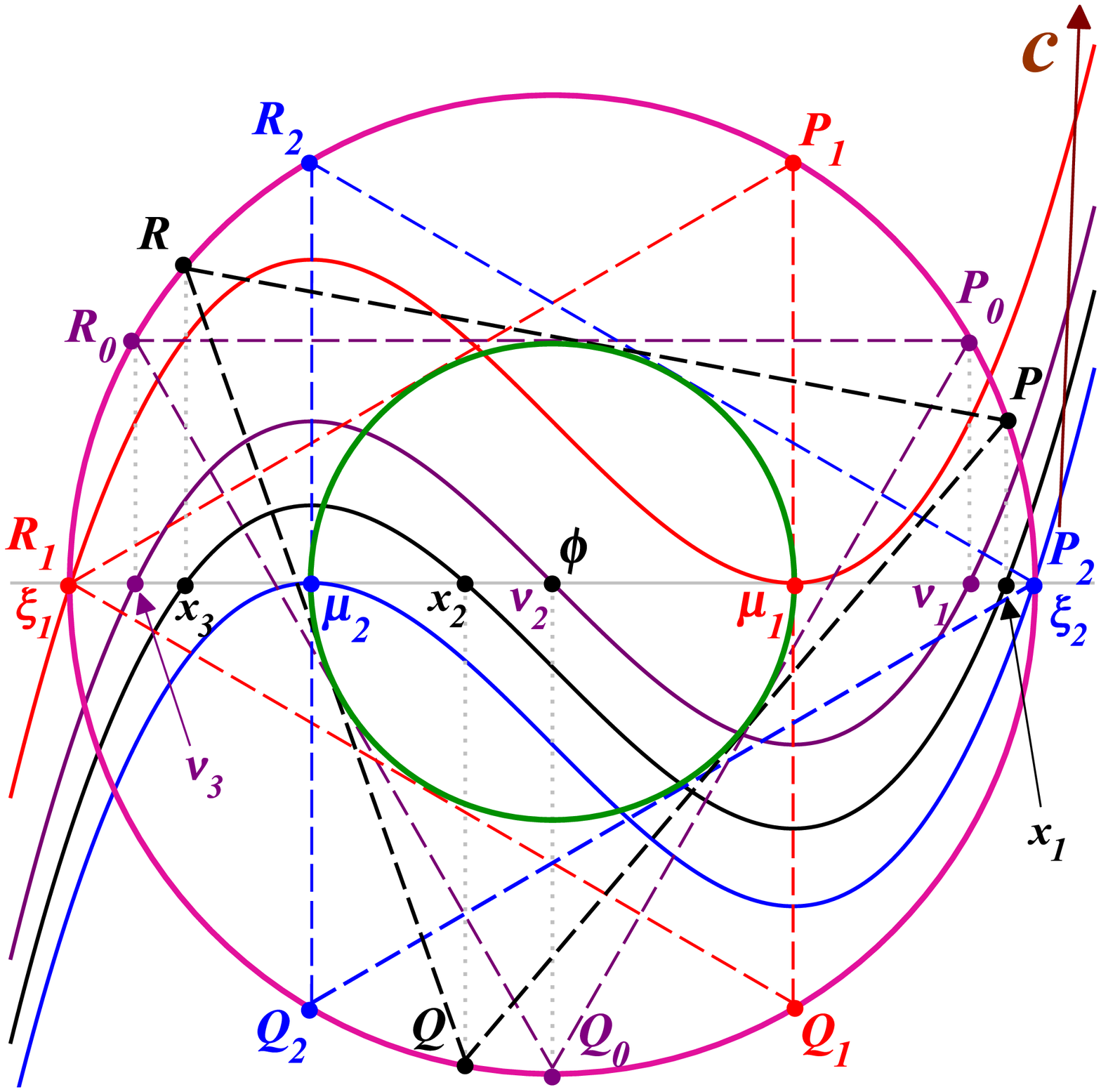}}
{\begin{minipage}{36em}
\scriptsize
\vskip.3cm
\begin{center}
{\bf Figure 2}
\end{center}
\vskip-.3cm
Presented here are four cubics and their Siebeck--Marden--Northshield triangles: the ``balanced" cubic with $c = c_0$ (second from top) whose roots are $\nu_{1,3}$, equidistant from $\phi = -a/3$, and $\nu_2 = \phi$ and whose triangle $P_0 Q_0 R_0$ has the side $P_0 R_0$ parallel to the abscissa; the two ``extreme" cubics --- with $c =  c_{1,2}$ (top and bottom) having double real roots $\mu_{1,2}$ and a simple root $\xi_{1,2}$ and whose triangles $P_{1,2} Q_{1,2} R_{1,2}$ have a side perpendicular to the abscissa and a vertex on the abscissa; and the general cubic (second from bottom)  $x^3 + a x^2 + b x + c$ with distinct real roots $x_3 < x_2 < x_1$ and triangle $PQR$. Increasing $c$ rotates the Siebeck--Marden--Northshield triangle counterclockwise about its centroid. Decreasing $c$ results in its clockwise rotation.  The isolation intervals of the roots of the latter can be immediately determined from the graph.
\end{minipage}
\end{figure}

\begin{proof}
Given that the root $x_2 = \nu_2 = \phi = -a/3$ of the ``balanced" cubic equation $x^3 + a x^2 + b x + c_0 = 0$, where $c_0 = - 2a^3/27 + a b/3$, is the midpoint between its other two roots $x_{1,3} = \nu_{1,3} = - a/3 \pm \sqrt{a^2/3 - b}$, one has $x_1 - x_2$ ($\sqrt{3}$ times the second coordinate of point $R$) being equal to $x_2 - x_3$ ($\sqrt{3}$ times the second coordinate of point $P$) --- see Figure 2. Hence $P$ and $R$ are both above the abscissa and are equidistant from it. Thus $PR$ is parallel to the abscissa. Hence, the distance between $x_3$ and $x_1$ is exactly equal to the length $\alpha = (\sqrt{12}/3) \sqrt{a^2 - 3b}$ of the side $PR$. In any other case of three real roots ($c \in [c_2, c_1]$ and $c \ne c_0$), the side $PR$ will not be parallel to the abscissa and hence the projection of $PR$ onto the abscissa will be shorter than the length of $PR$, that is, the three real roots of the cubic polynomial will lie in an interval of length smaller than $\alpha = (\sqrt{12}/3) \sqrt{a^2 - 3b}$.
\end{proof}
\n
Note that the Siebeck--Marden--Northshield triangle rotates counter-clockwise when increasing the free term $c$ and clockwise otherwise. The triangle cannot be rotated counter-clockwise or clockwise further than the triangles of the two ``extreme" cubics (with $c =  c_{1,2}$) as three real roots exist and, hence, the Siebeck--Marden--Northshield triangle exists itself, only for $c \in [c_2, c_1]$. \\
Also observe a completely geometric in nature proof that the projection of the incircle of the Siebeck--Marden--Northshield triangle coincides exactly with the interval given by the two critical points of the cubic: the incircle is invariant when varying the free term $c$ from $c_2$ to $c_1$ and this variation moves the graph up from the position of a local maximum tangent to the abscissa --- the ``extreme" cubic with $c = c_2$ (the lowermost curve on Figure 2) to a local minimum tangent to the abscissa --- the ``extreme" cubic with  $c = c_1$ (the uppermost curve on Figure 2), whose triangles are $P_{2,1} Q_{2,1} R_{2,1}$, respectively.

\begin{theorem}
The monic cubic polynomial $p(x) = x^3 + a x^2 + b x + c$, for which $b < a^2/3$ and $c \in [c_2, c_1]$,
has three real roots $x_3 \le x_2 \le x_1$, at least two of which are different and any two of which are not farther apart than $(\sqrt{12}/3) \sqrt{a^2 - 3b}$, with the following isolation intervals:
\begin{itemize}
\item [(I)] For $c_2 \le c \le c_0$: $x_3 \in [\nu_3, \mu_2], \,\, x_2 \in [\mu_2, \phi]$, and  $x_1 \in [\nu_1, \xi_2]$.

\item [(II)] For $c_0 \le c \le c_1$: $x_3 \in [\xi_1, \nu_3], \,\, x_2 \in [\phi,\mu_1]$, and  $x_1 \in [\mu_1, \nu_1]$,

\end{itemize}
where:
\begin{itemize}
\item [(i)] $\mu_{1,2}$ is the double root and $\xi_{1,2}$ is the simple root of $p_{1,2}(x) = x^3 + a x^2 + b x + c_{1,2}$, that is, $\mu_{1,2}$ are the roots of $p'(x) = 3 x^2 + 2 a x + b = 0$, namely: $\mu_{1,2} = -a/3 \pm r = -a/3 \pm (1/3) \sqrt{a^2 - 3 b}$ and $\xi_{1,2} =  - a - 2 \mu_{1,2} = - a/3 \mp 2r = -a/3 \mp (2/3) \sqrt{a^2 - 3 b}$.

\item [(ii)] $\nu_{1,2,3}$ are the roots of the ``balanced" cubic equation $p_0(x) = x^3 + a x^2 + b x + c_0$, namely: $\nu_{1,3} = -a/3 \pm \alpha/2 = -a/3 \pm (\sqrt{3}/3) \sqrt{a^2 - 3 b}$ and $\nu_2 = \phi = - a/3$.
\end{itemize}
\end{theorem}

\begin{proof}
Due to Lemma 2, the discriminant $\Delta_3 = -27 c^2 + (18 a b - 4 a^3) c + a^2 b^2 - 4b^3$ of the monic cubic polynomial $x^3 + a x^2 + b x + c$ is
non-negative for all $a$ and $b \le a^2/3$, if $c$ is between the roots $c_{1,2} = c_0 \pm (2/27) \sqrt{( a^2 - 3 b)^3}$ (with $c_0= - 2 a^3/27 + a b/3$) of the quadratic equation $x^2 + (4 a^3/27 - 2 a b/3 ) x  - a^2 b^2/27  + 4 b^3/27 = 0$. Then $x^3 + a x^2 + b x + c$ will have three real roots. The two ``extreme" cases, the cubics $x^3 + a x^2 + b x + c_1$ and $x^3 + a x^2 + b x + c_2$, will each have a double root (as $\Delta_3$ vanishes for $c = c_{1,2}$) and a simple root. Otherwise, for $c_2 < c < c_1$, the cubic polynomial will have three distinct roots. \\
If $\mu_{1,2}$ is the double root of the ``extreme" cubic $x^3 + a x^2 + b x + c_{1,2}$ and $\xi_{1,2}$ --- the corresponding simple root, then, when $c = c_{1,2}$, one has (due to Vi\`ete formul\ae): $2 \mu_i + \xi_i = - a, \,\, \mu_i^2 + 2 \mu_i \xi_i = b,$ and $\mu_i^2 \xi_i = -c$ (for $i = 1, 2$).
Expressing from the first $\xi_i = - a - 2 \mu_i$ and substituting into the second yields $-3 \mu_i^2 - 2 a \mu_i - b = 0$, that is, the double roots $\mu_{1,2}$ of each of the ``extreme" cubics  $x^3 + a x^2 + b x + c_{1,2}$ are the roots of the quadratic equation $3x^2 + 2 a x + b = 0$, that is
$\mu_{1,2} = -a/3 \pm r = -a/3 \pm (1/3) \sqrt{a^2 - 3 b}$. Hence one finds: $\xi_{1,2}  =  - a - 2 \mu_{1,2} = - a/3 \mp 2r = -a/3 \mp (2/3) \sqrt{a^2 - 3 b}$. \\
Due to Lemma 3, the biggest distance between the roots of the cubic will be $\alpha = (\sqrt{12}/3) \sqrt{a^2 - 3b}$. \\
The roots of the ``balanced" cubic equation $x^3 + a x^2 + b x - 2a^3/27 + a b/3 = 0$ (see the proof of Lemma 3) are symmetric with respect to the centre of the inscribed circle: $\nu_3 = -a/3 - \sqrt{a^2/3 - b}$, $\nu_2 = \phi = -a/3$, and $\nu_1 = - a/3 + \sqrt{a^2/3 - b}$. The ``balanced" equation has triangle $P_0 Q_0 R_0$ and the side $P_0 R_0$ is parallel to the abscissa (Figure 2). \\
When $c = c_1 > c_0$, the Siebeck--Marden--Northshield triangle is $P_1 Q_1 R_1$ and its side $P_1 Q_1$ is perpendicular to the abscissa. Hence the roots $x_2$ and $x_1$ coalesce into the double root $\mu_1$. The vertex $R_1$ is on the abscissa at the smallest root $\xi_1$ (Figure 2). \\
When $c = c_2 < c_0$, the Siebeck--Marden--Northshield triangle is $P_2 Q_2 R_2$ and its side $R_2 Q_2$ is perpendicular to the abscissa. The roots $x_3$ and $x_2$ coalesce into the double root $\mu_2$, while the biggest root $x_1$ is equal to $\xi_2$, as the vertex $P_2$ is on the abscissa at $\xi_2$ (Figure 2). \\
The isolation intervals of the roots of the cubic polynomial are then easily read geometrically --- see Figure 2.
\end{proof}

\n
The lengths of the isolation intervals of the roots are as follows:
\begin{itemize}
\item [(I)] \underline{\bm{$c_2 \le c \le c_0$}} \\
For the smallest root $x_3$, the length is $\mu_2 - \nu_3 = [(\sqrt{3} - 1)/3 ] \, \sqrt{a^2 - 3b}$; for the middle root $x_2$ one has $\phi - \mu_2 = (1/3) \sqrt{a^2 - 3b}$; and for the largest root $x_1$ it is $\xi_2 - \mu_1 = [(2 - \sqrt{3})/3 ] \, \sqrt{a^2 - 3b}$.
\item [(II)] \underline{\bm{$c_0 \le c \le c_1$}} \\
For the smallest root $x_3$, the length is $\mu_3 - \xi_1 = [(2 - \sqrt{3})/3 ] \, \sqrt{a^2 - 3b}$; for the middle root $x_2$ one has $\phi - \mu_2 = (1/3) \sqrt{a^2 - 3b}$; and for the largest root $x_1$ it is $\xi_2 - \mu_1 = [(\sqrt{3} - 1)/3 ] \, \sqrt{a^2 - 3b}$.
\end{itemize}
\begin{theorem}
The monic cubic polynomial $p(x) = x^3 + a x^2 + b x + c$, for which $b < a^2/3$ and:
\begin{itemize}
\item [(I)] $c < c_2$, has only one real root: $x_1 > \xi_2 = - a - 2 \mu_2 = - a/3 + 2r  = - a/3 + (2/3) \sqrt{a^2 - 3 b}$ (it can be bounded from above by a polynomial root bound);
\item [(II)] $c > c_1$, has only one real root: $x_1 < \xi_1 =  - a - 2 \mu_1 = - a/3 - 2r  = - a/3 - (2/3) \sqrt{a^2 - 3 b}$ (it can be bounded from below by a polynomial root bound).
\end{itemize}
\end{theorem}

\begin{proof}
Given on Figure 3 are the two ``extreme" cubics --- with $c = c_1$ (second from top) and with $c = c_2$ (second from bottom). Their corresponding triangles are $P_1 Q_1 R_1$ and $P_2 Q_2 R_2$, respectively. Each of these cubics has a double root $\mu_{1,2}$ and a simple root $\xi_{1,2}$, respectively. Cubics with $c$ such that $c_2 < c < c_1$ are between those two and they are the only ones with three distinct real roots. When $c > c_1$ (uppermost cubic), there is a pair of complex conjugate roots and a single real  root $x_1 < \xi_1 = - a/3 - (2/3) \sqrt{a^2 - 3 b}$. When $c < c_2$ (lowermost cubic), there is a pair of complex conjugate roots and a single real root $x_1 > \xi_2 = - a/3 + (2/3) \sqrt{a^2 - 3 b}$. The isolation intervals of the single real root for either of the two latter cubics can be found by the determination of the lower (respectively, upper) root bound of the cubic.
\end{proof}

\n
As polynomial upper root bound, one can take one of the many existing root bounds. For example, it could be the bigger of 1 and the sum of the absolute values of all negative coefficients \cite{1}. Or one can consider the bound \cite{kur}:  $1 + \sqrt[k]{H}$, where $k = 1$ if $a < 0, \,\, k = 2$ if $a > 0$ and $b <0$, and $k = 3$ if $a > 0$ and $b > 0,$ and $c < 0$ (if $a$, $b$, and $c$ are all positive, the upper root bound is zero). $H$ is the biggest absolute value of all negative coefficients in $x^3 + a x^2 + b x + c$. \\
The lower root bound is the negative of the upper root bound of $-x^3 + a x^2 - b x + c$.

\begin{center}
\begin{tabular}{cc}
\includegraphics[width=67mm]{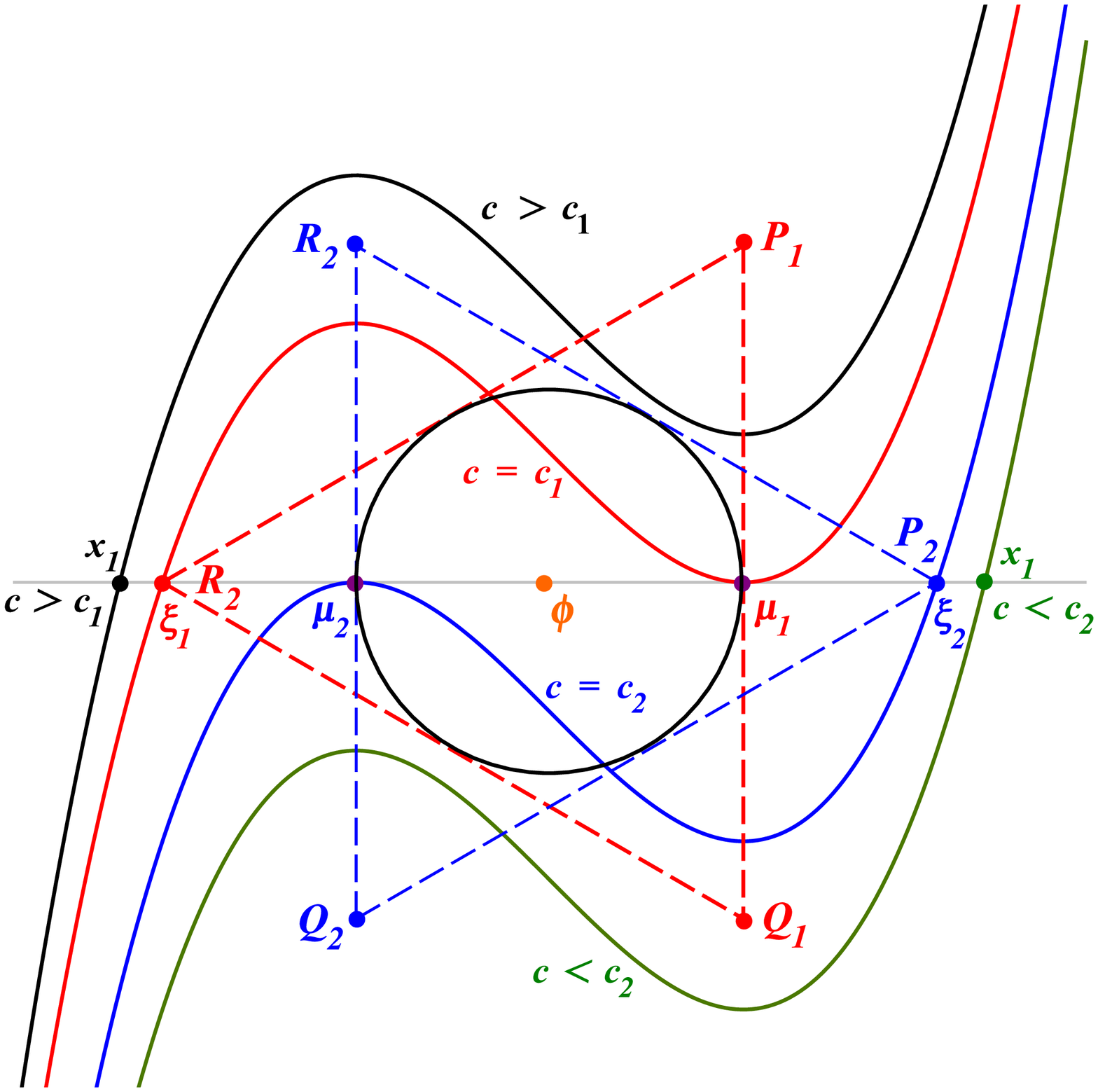} & \includegraphics[width=67mm]{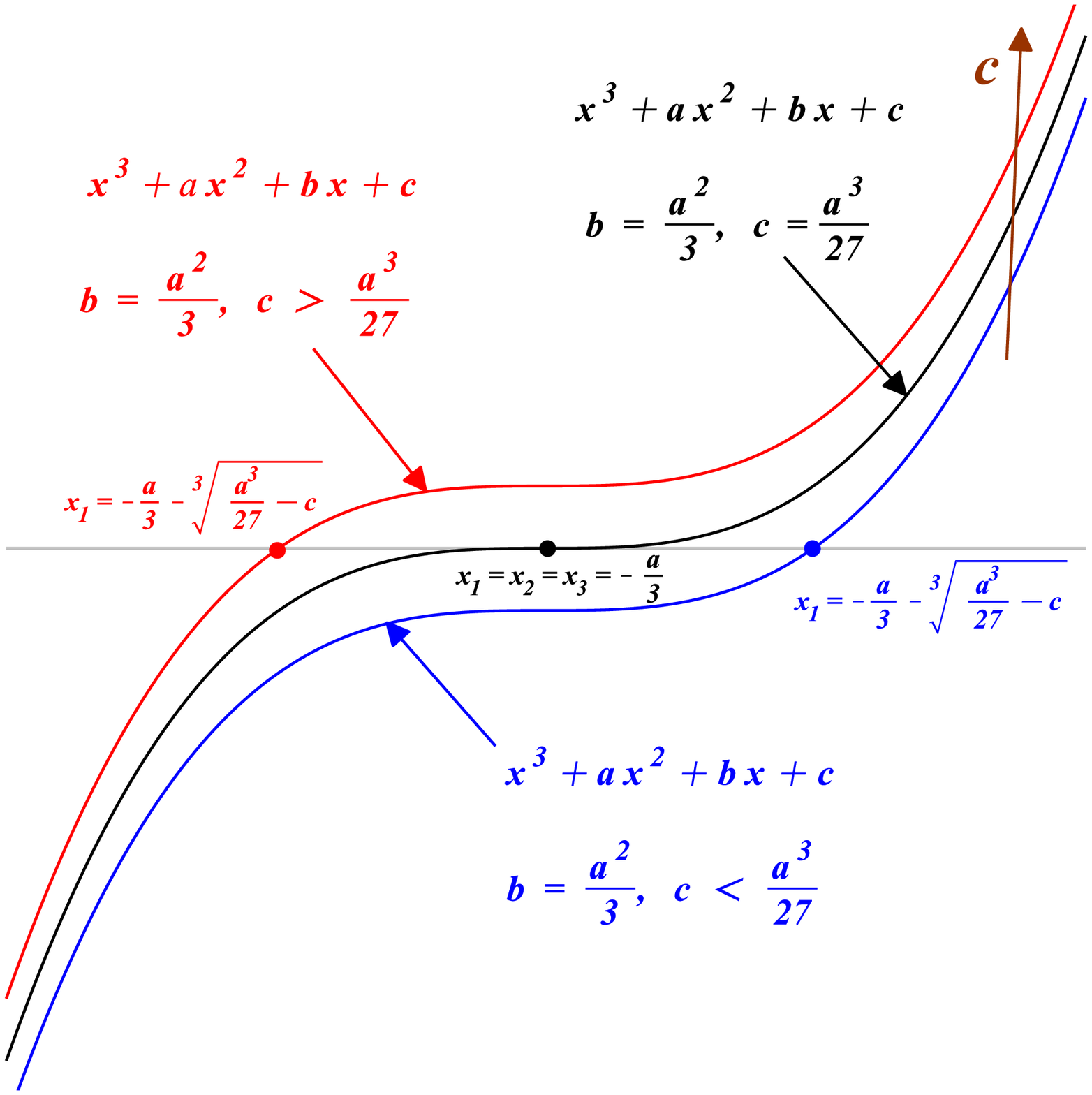} \\
{\scriptsize {\bf Figure 3}} &  {\scriptsize {\bf Figure 4}} \\
{\scriptsize {\it Theorem 2}} & {\scriptsize {\it Theorem 3}} \\
& \\
\multicolumn{1}{c}{\begin{minipage}{18em}
\scriptsize
\vskip-0.74cm
When $b < a^2/3$ and:
\begin{itemize}
\item [(I)] $c < c_2$, the cubic has only one real root: $x_1 > \xi_2 = - a - 2 \mu_2 = - a/3 + 2r  = - a/3 + (2/3) \sqrt{a^2 - 3 b}$;
\item [(II)] $c > c_1$, the cubic has only one real root: $x_1 < \xi_1 =  - a - 2 \mu_1 = - a/3 - 2r  = - a/3 - (2/3) \sqrt{a^2 - 3 b}$.
\end{itemize}
\end{minipage}}
& \multicolumn{1}{c}{\begin{minipage}{18em}
\scriptsize
\vskip-.4cm
When $b = a^2/3$ and:
\begin{itemize}
\item [(I)] $c < (1/27) a^3$, the cubic has only one real root: $x_1 = -a/3 + \sqrt[3]{a^3/27 - c} \, > \, -a/3$;
\item [(II)] $c = (1/27) a^3$, the cubic has a triple real root: $x_1 = x_2 = x_3 = -a/3$;
\item [(III)] $c > (1/27) a^3$, the cubic has only one real root: $x_1 = -a/3 + \sqrt[3]{a^3/27 - c} \, < \, -a/3$.
\end{itemize}
\end{minipage}}
\\
\end{tabular}
\end{center}

\begin{theorem}
The monic cubic polynomial $p(x) = x^3 + a x^2 + b x + c$, for which $b = a^2/3$ and:
\begin{itemize}
\item [(I)] $c < (1/27) a^3$, has only one real root: $x_1 = -a/3 + \sqrt[3]{a^3/27 - c} \, > \, -a/3$;
\item [(II)] $c = (1/27) a^3$, has a triple real root: $x_1 = x_2 = x_3 = -a/3$;
\item [(III)] $c > (1/27) a^3$, has only one real root: $x_1 = -a/3 + \sqrt[3]{a^3/27 - c} \, < \, -a/3$.
\end{itemize}
\end{theorem}

\begin{proof}
Shown on Figure 4 is the special case of $b = a^2/3$. One immediately gets that $c_1 = c_2 = a^3/27$ in this case. The only cubic with three real roots is the one with $c = a^3/27$. This is the cubic $x^3 + a x^2 + (a^2/3) x + a^3/27 = (x + a/3)^3$ (middle curve). Clearly, this cubic has a triple real root $x_1 = x_2 = x_3 = - a/3$. If one increases $c$ above $a^3/27$ (top cubic), there is a pair of complex conjugate roots and a single root $x_1 < - a/3$. If one increases $c$ above $a^3/27$ (bottom cubic), there is a pair of complex conjugate roots and a single root $x_1 > - a/3$. The single real root for either of the two latter cubics can be immediately found completing the cube: $x^3 + ax^2 + (a^2/3)x + c = (x + a/3)^3 - a^3/27 + c$. Hence, $x_1 = -a/3 + \sqrt[3]{a^3/27 - c}$.
\end{proof}

\begin{theorem}
The only real root $x_1$ of the monic cubic polynomial $p(x) = x^3 + a x^2 + b x + c$ with $b > a^2/3$ (due to Lemma 1) has the following isolation interval:
\begin{itemize}
\item [(I)] If $a \ge 0$ and $c \le 0: \,\, 0 \le x_1 \le -c/b$.
\item [(II)] If $a \ge 0$ and $c > 0: \,\,$  min$\{-a, -c/b\} \le x_1 \le $ max$\{-a, -c/b\}$.
\item [(III)] If $a < 0$ and $c < 0: \,\,$  min$\{-a, -c/b\} \le x_1 \le $ max$\{-a, -c/b\}$.
\item [(IV)] If $a < 0$ and $c \ge 0: \,\, -c/b \le x_1 \le 0$.
\end{itemize}
\end{theorem}

\begin{proof}
Re-write the cubic equation $x^3 + a x^2 + b x + c = 0$ as $x^3 + a x^2 = - b x - c$. Such ``split" of polynomial equations of different degrees has been proposed and studied in \cite{2, 3, 4} \\
The rest of the proof is graphic --- see the captions of Figures 5--8 for the four cases {\it (I)--(IV)} respectively.
\end{proof}
\begin{center}
\begin{tabular}{cc}
\includegraphics[width=40mm]{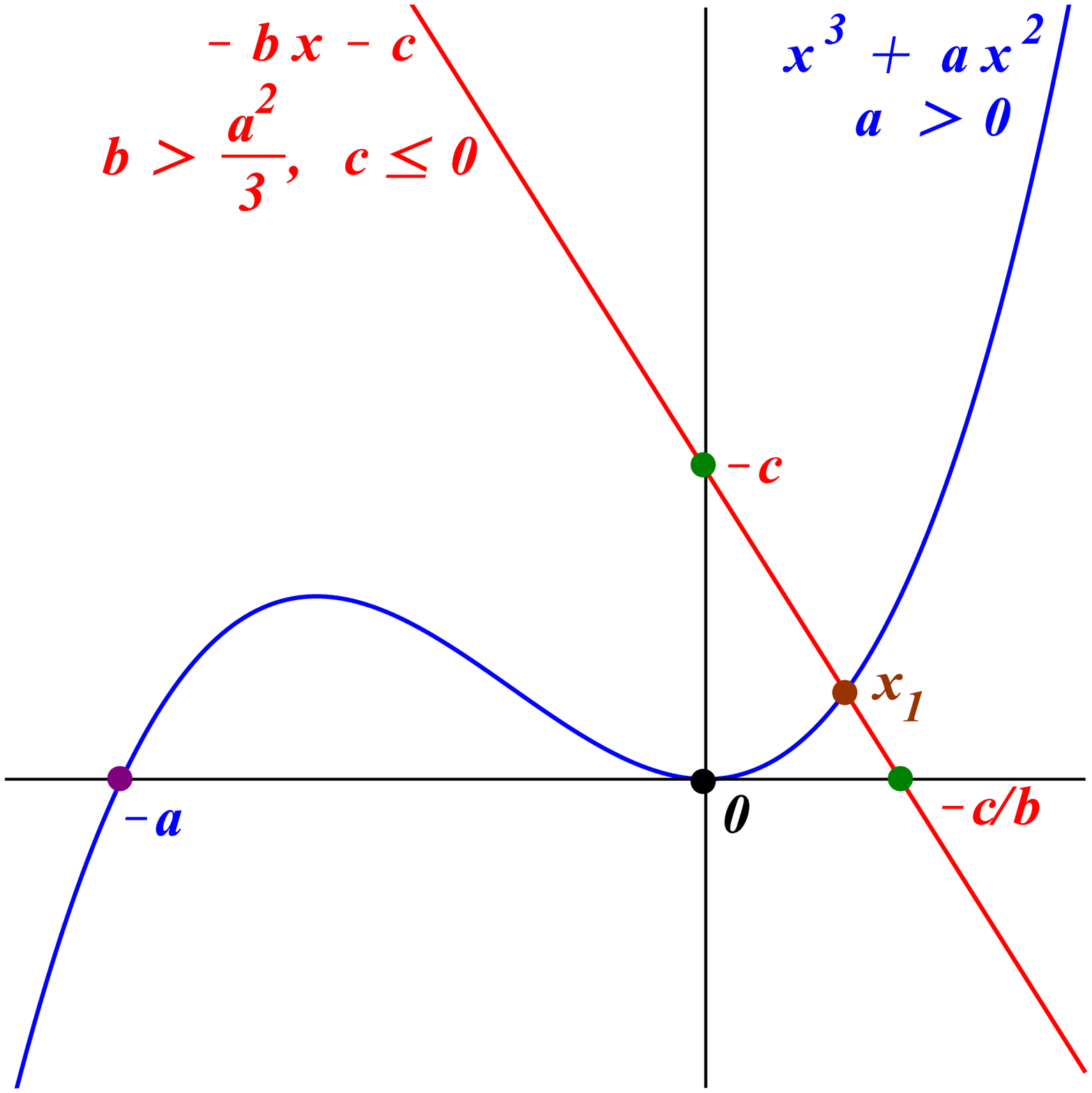} & \includegraphics[width=40mm]{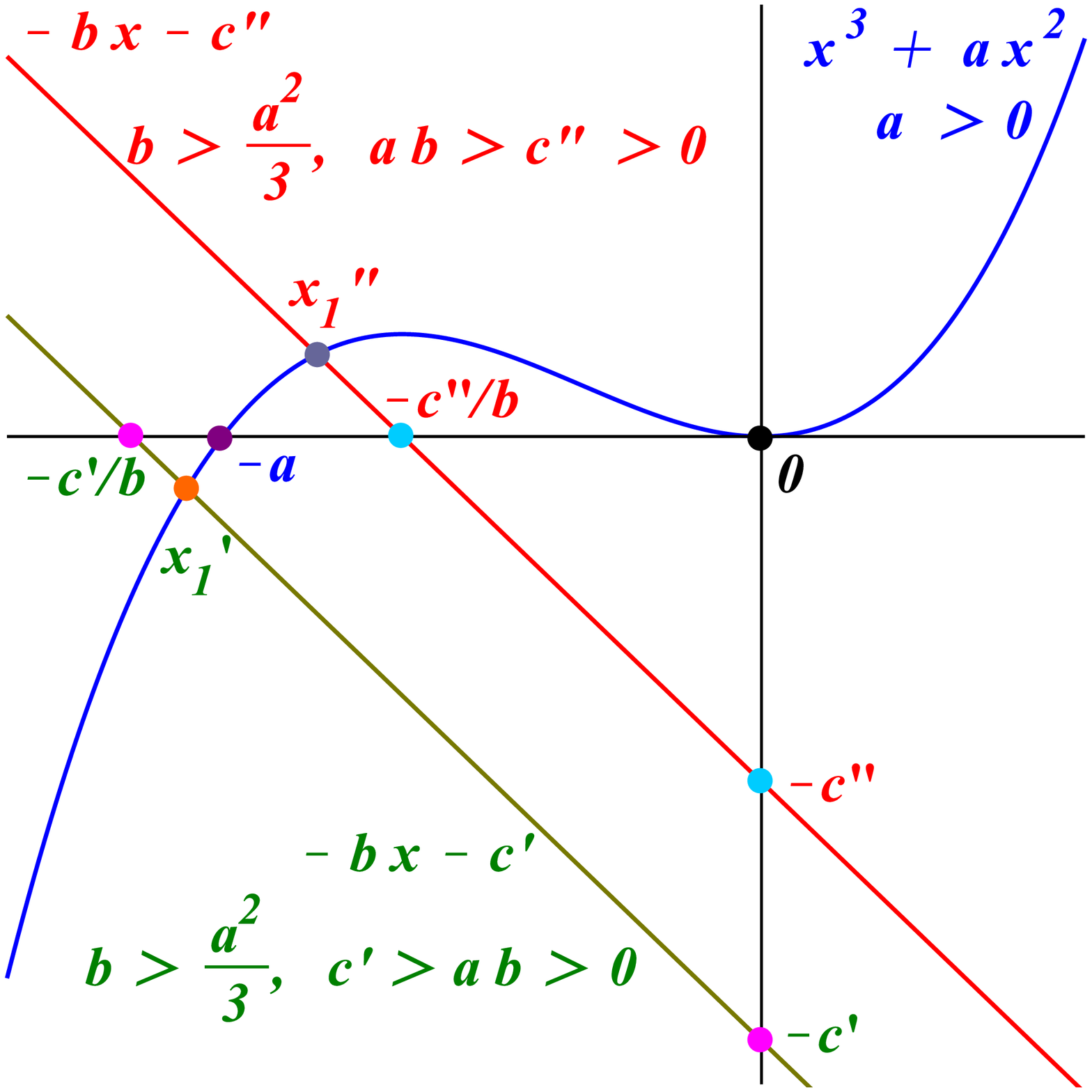} \\
{\scriptsize {\bf Figure 5}} &  {\scriptsize {\bf Figure 6}} \\
{\scriptsize {\it Proof of Theorem 4(I)}} & {\scriptsize {\it Proof of Theorem 4(II)}} \\
& \\
\multicolumn{1}{c}{\begin{minipage}{18em}
\scriptsize
\vskip-.45cm
When $a \ge 0$ and $c \le 0$, the isolation interval of the single root $x_1$ is: $0 \le x_1 \le -c/b$.
\end{minipage}}
& \multicolumn{1}{c}{\begin{minipage}{18em}
\scriptsize
\vskip-.15cm
When $a \ge 0$ and $c > 0$, the isolation interval of the single root $x_1$ is: min$\{-a, -c/b\} \le x_1 \le $ max$\{-a, -c/b\}$.
\end{minipage}}
\\
\\
\end{tabular}
\end{center}

\begin{center}
\begin{tabular}{cc}
\includegraphics[width=40mm]{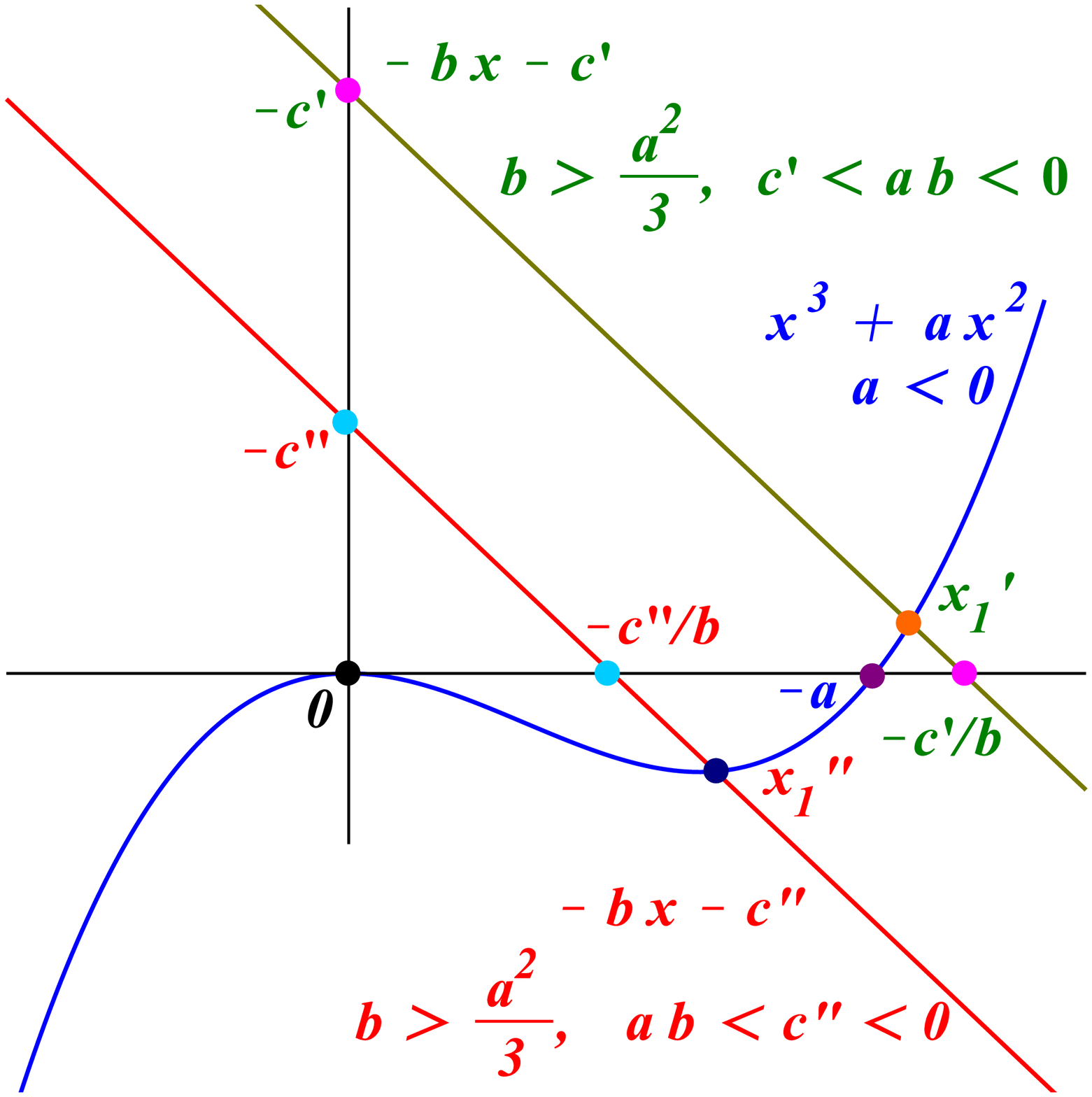} & \includegraphics[width=40mm]{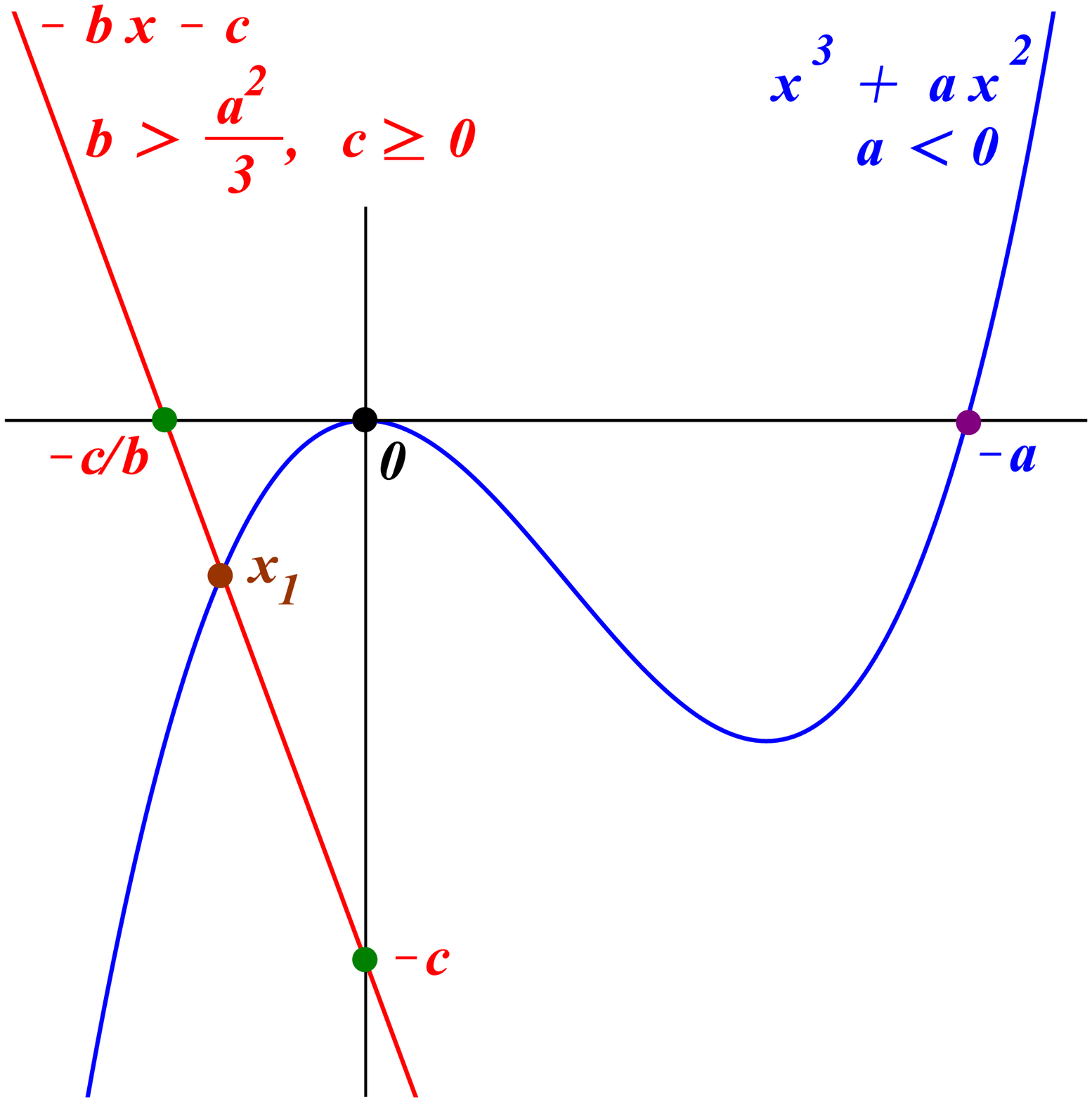} \\
{\scriptsize {\bf Figure 7}} &  {\scriptsize {\bf Figure 8}} \\
{\scriptsize {\it Proof of Theorem 4(III)}} & {\scriptsize {\it Proof of Theorem 4(IV)}} \\
& \\
\multicolumn{1}{c}{\begin{minipage}{18em}
\scriptsize
\vskip-.15cm
When $a < 0$ and $c < 0$, the isolation interval of the single root $x_1$ is: min$\{-a, -c/b\} \le x_1 \le $ max$\{-a, -c/b\}$.
\end{minipage}}
& \multicolumn{1}{c}{\begin{minipage}{18em}
\scriptsize
\vskip-.45cm
When $a < 0$ and $c \ge 0$, the isolation interval of the single root $x_1$ is: $0 \le x_1 \le -c/b$.
\end{minipage}}
\\
\\
\end{tabular}
\end{center}

\section{Roles of the Coefficients and Root Isolation Intervals --- Summary and Application of the Analysis}
\begin{itemize}
\item[\bf{(a)}] The coefficient $a$ of the quadratic term of $x^3 + a x^2 + b x + c$ selects the centre $\phi = - a/3$ of the inscribed circle of the equilateral triangle that projects onto the roots of $x^3 + a x^2 + b x + c$, in the case of three real roots. The centre of this circle is also the projection of the inflection point of the graph of $x^3 + a x^2 + b x + c$ onto the abscissa. The inscribed circle projects to an interval on the abscissa with endpoints equal to the projections of the stationary points of $x^3 + a x^2 + b x + c$ (Figure 1).

\item[\bf{(b)}] For any given $a$, the coefficients $b$ of the linear term of $x^3 + a x^2 + b x + c$ determines the radius $r = (1/3) \sqrt{a^2 - 3b}$ of the inscribed circle. The circumscribed circle of the equilateral triangle has radius $2 r = (2/3) \sqrt{a^2 - 3b}$. \\
    If a cubic polynomial has two stationary points, the distance between them is always $2 r = (2/3) \sqrt{a^2 - 3b}$. \\
    The inflection point of the graph of $x^3 + a x^2 + b x + c$ is always the midpoint ($-a/3$) between the stationary points of the cubic polynomial. \\
    {\it Hence, the analysis of the cubic polynomial $x^3 + a x^2 + b x + c$ should start with what the value of $b$, relative to $a^2/3$, is.}
\begin{itemize}
  \item[\bf{(I)}] If $\bm{b < a^2/3}$ and if:
\begin{itemize}
\item[\bf{(i)}] $\bm{c_2 \le c \le c_0}$, then the polynomial $x^3 + a x^2 + b x + c$ has three real roots with the following isolation intervals: $x_3 \in [\nu_3, \mu_2], \,\, x_2 \in [\mu_2, \phi]$, and  $x_1 \in [\nu_1, \xi_2]$ (Figure 2).
\item[\bf{(ii)}] $\bm{c_0 \le c \le c_1}$, then the polynomial $x^3 + a x^2 + b x + c$ has three real roots with the following isolation intervals: $x_3 \in [\xi_1, \nu_3], \,\, x_2 \in [\phi,\mu_1]$, and  $x_1 \in [\mu_1, \nu_1]$ (Figure 2).
\end{itemize}
In the above, $c_{1,2} = c_0  \pm  (2/27) \sqrt{( a^2 - 3 b)^3}$, with $c_0 = - 2a^3/27 + a b/3$, are the values of $c$ for which, for any $a$ and $b < a^2/3$, the discriminant $\Delta_3$ of the cubic polynomial $x^3 + a x^2 + b x + c$ is zero ($\Delta_3$ positive for $c$ between $c_2$ and $c_1$). Namely, these are the roots of the quadratic equation (\ref{quadratic}): $x^2 + (4 a^3/27 - 2 a b/3 ) x  - a^2 b^2/27  + 4 b^3/27 = 0$. \\
Also in the above,  $\nu_3 = -a/3 - \sqrt{a^2/3 - b}$, $\nu_2 = \phi = -a/3$, and $\nu_1 = - a/3 + \sqrt{a^2/3 - b}$ are three real roots of the ``balanced" cubic polynomial $x^3 + a x^2 + b x + c_0$ (Figure 2). \\
The roots of the ``extreme" cubic $x^3 + a x^2 + b x + c_1$ are the double root $\mu_1 = -a/3 + (\sqrt{3}/3) \, \sqrt{a^2/3-b}$  and the simple root $\xi_1 =  - a - 2 \mu_1 = - a/3 - 2r  = - a/3 - (2/3) \sqrt{a^2 - 3 b}$. Likewise, the roots of the ``extreme" cubic $x^3 + a x^2 + b x + c_1$ are the double root $\mu_2  = -a/3 - (\sqrt{3}/3) \, \sqrt{a^2/3-b}$ and the simple root $\xi_2 = - a - 2 \mu_2 = - a/3 + 2r  = - a/3 + (2/3) \sqrt{a^2 - 3 b}$ (Figure 2 and Figure 3). \\
The biggest distance between any two of the three real roots of the cubic equation $x^3 + a x^2 + b x + c = 0$ is $\alpha = \sqrt{12} r = (\sqrt{12}/3) \sqrt{a^2 - 3b}$ --- achieved for the roots of the ``balanced" cubic equation $x^3 + a x^2 + b x + c_0$ (Figure 2). \\
For any other cubic equation with $c_2 \le c \le c_1$, the three real roots are within an interval of length $3 r = \sqrt{a^2 - 3b} < \alpha$ (Figure 2).
\begin{itemize}
\item[\bf{(iii)}] $\bm{c < c_2}$, then the polynomial $x^3 + a x^2 + b x + c$ has only one real root: $x_1 > \xi_2 = - a - 2 \mu_2 = - a/3 + 2r  = - a/3 + (2/3) \sqrt{a^2 - 3 b}$ (Figure 3). The root $x_1$ can be bounded from above by a polynomial root bound.
\item[\bf{(iv)}] $\bm{c > c_1}$, then the polynomial $x^3 + a x^2 + b x + c$ has only one real root: $x_1 < \xi_1 =  - a - 2 \mu_1 = - a/3 - 2r  = - a/3 - (2/3) \sqrt{a^2 - 3 b}$ (Figure 3). The root $x_1$ can be bounded from below by a polynomial root bound.
\end{itemize}

\item[\bf{(II)}] If $\bm{b = a^2/3}$ and if:
\begin{itemize}
    \item[\bf{(i)}] $\bm{c < (1/27) a^3}$, then the polynomial $x^3 + a x^2 + b x + c$ has only one real root: $x_1 = -a/3 + \sqrt[3]{a^3/27 - c} \, > \, -a/3$ (Figure 4).
    \item [\bf{(ii)}] $\bm{c = (1/27) a^3}$, then the polynomial $x^3 + a x^2 + b x + c$ has a triple real root: $x_1 = x_2 = x_3 = -a/3$ (Figure 4).
    \item [\bf{(iii)}] $\bm{c > (1/27) a^3}$, then the polynomial $x^3 + a x^2 + b x + c$ has only one real root: $x_1 = -a/3 + \sqrt[3]{a^3/27 - c} \, < \, -a/3$ (Figure 4).
\end{itemize}
\item[\bf{(III)}] If $\bm{b > a^2/3}$, the discriminant of the cubic polynomial is negative and thus $x^3 + a x^2 + b x + c$  has one real root $x_1$ and a pair of complex conjugate roots. The isolation interval of $x_1$ depends on the signs of $a$ and $c$ and is as follows:
\begin{itemize}
\item [\bf{(i)}] If $\bm{a \ge 0}$ and $\bm{c \le 0}\!: \,\, 0 \le x_1 \le -c/b$ (Figure 5).
\item [\bf{(ii)}] If $\bm{a \ge 0}$ and $\bm{c > 0}\!: \,\,$  min$\{-a, -c/b\} \le x_1 \le $ max$\{-a, -c/b\}$ (Figure 6).
\item [\bf{(iii)}] If $\bm{a < 0}$ and $\bm{c < 0}\!: \,\,$  min$\{-a, -c/b\} \le x_1 \le $ max$\{-a, -c/b\}$ (Figure 7).
\item [\bf{(iv)}] If $\bm{a < 0}$ and $\bm{c \ge 0}\!: \,\, -c/b \le x_1 \le 0$ (Figure 8).
\end{itemize}

\end{itemize}

\item[\bf{(c)}] The coefficient $c$ of $x^3 + a x^2 + b x + c$ rotates the equilateral triangle (which exists if $b < a^2/3$) that projects onto the roots $x_3 \le x_2 \le x_1$ (at least two of which are different) of the cubic polynomial. The vertices $P$, $Q$, and $R$ of the triangle are points of coordinates $(x_1, (x_2 - x_3)/\sqrt{3})$, $(x_2, (x_3 - x_1)/\sqrt{3})$, and $(x_3, (x_1 - x_2)/\sqrt{3})$, respectively. Point $Q$ is always below the abscissa and points $P$ and $R$ --- always above it. \\
    When $c = c_0 = - 2a^3/27 + a b/3$, the side $PR$ is parallel to the abscissa. This corresponds to the ``balanced" cubic equation $x^3 + a x^2 + b x - 2a^3/27 + a b/3 = 0$, the roots of which are symmetric with respect to the centre of the inscribed circle: $\nu_3 = -a/3 - \sqrt{a^2/3 - b}$, $\nu_2 = \phi = -a/3$, and $\nu_1 = - a/3 + \sqrt{a^2/3 - b}$. The ``balanced" equation has triangle $P_0 Q_0 R_0$ (Figure 2). \\
    When $c$ increases from $c_0$ towards $c_1 > c_0$, the equilateral triangle $PQR$ rotates counterclockwise around its centre from the position of triangle $P_0 Q_0 R_0$ of the ``balanced" equation. When $c = c_1$, the roots $x_2$ and $x_1$ coalesce into the double root $\mu_1$, while the smallest root $x_3$ becomes equal to $\xi_1 =  - a - 2 \mu_1 = - a/3 - 2r  = - a/3 - (2/3) \sqrt{a^2 - 3 b}$. The triangle in this case is $P_1 Q_1 R_1$ and its side $P_1 Q_1$ is perpendicular to the abscissa. The vertex $R_1$ is on the abscissa. The triangle cannot be rotated further counterclockwise as, when $c > c_1$, the polynomial $x^3 + a x^2 + b x + c$ has only one real root (Figure 2). \\
    When $c$ decreases from $c_0$ towards $c_2 < c_0$, the equilateral triangle $PQR$ rotates clockwise around its centre from the position of triangle $P_0 Q_0 R_0$ of the ``balanced" equation. When $c = c_2$, the roots $x_3$ and $x_2$ coalesce into the double root $\mu_2$, while the biggest root $x_1$ becomes equal to $\xi_2 = - a - 2 \mu_2 = - a/3 + 2r  = - a/3 + (2/3) \sqrt{a^2 - 3 b}$. The triangle in this case is $P_2 Q_2 R_2$ and its side $R_2 Q_2$ is perpendicular to the abscissa. The vertex $P_2$ is on the abscissa. The triangle cannot be rotated further clockwise as, when $c < c_2$, the polynomial $x^3 + a x^2 + b x + c$ has only one real root (Figure 2).
\end{itemize}

\section{Examples}

\n
Each possible case --- for each Theorem (1 to 4, with the relevant subsection of the Theorem given in brackets in Roman numerals) --- is illustrated with an example. The roots of the cubics in these examples are found numerically with Maple 2021.
\begin{enumerate}
\item {\bf Theorem 1(I)}, $\bm{b < a^2/3, \,\, c_2 \le c \le c_0}$: \,\,\, \underline{$x^3 +3x^2 + 2x - \frac{1}{4} = 0$}.
\vskip.01cm
One has: $c_0 = 0, \,\, c_1 = 0.385, \,\, c_2 = -0.385$. Also: $\mu_1 = -0.423, \, \, \mu_2 = -1.577, \,\, \nu_1 = 0, \,\, \nu_2 = \phi = -1, \,\, \nu_3 = -2, \,\, \xi_1 = -2.155,$ and $\xi_2 = 0.155$. \\
The root isolation intervals are: $\nu_1 \le x_1 \le \xi_2, \,\,  \mu_2 \le x_2 \le \phi, $ and $\nu_3 \le x_3 \le \mu_2$, that is: $0 \le x_1 \le 0.155, \,\, -1.577 \le x_2 \le -1, $ and $-2 \le x_3 \le -1.577$. \\
The roots are: $x_1 = 0.107, \,\, x_2 = -1.270, $ and $x_3 = -1.840$.

\item {\bf Theorem 1(II)}, $\bm{b < a^2/3, \,\, c_0 \le c \le c_1}$: \,\,\, \underline{$x^3 - 4x^2 + 2x + 3 = 0$}. \\
One has: $c_0 = 2.074, \,\, c_1 = 4.416, \,\, c_2 = -0.268$. Also: $\mu_1 = 2.387, \, \, \mu_2 = 0.279, \,\, \nu_1 = 3.158, \,\, \nu_2 = \phi = 1.333, \,\, \nu_3 = -0.492, \,\, \xi_1 = -0.775,$ and $\xi_2 = 3.441$. \\
The root isolation intervals are: $\mu_1 \le x_1 \le \nu_1, \,\, \phi \le x_2 \le \mu_1, $ and $\xi_1 \le x_3 \le \nu_3$, that is: $2.387 \le x_1 \le 3.158, \,\, 1.333 \le x_2 \le 2.387, $ and $-0.775 \le x_3 \le -0.492$. \\
The roots are: $x_1 = 3, \,\, x_2 = 1.620, $ and $x_3 = -0.618$.

\item {\bf Theorem 2(I)}, $\bm{b < a^2/3, \,\, c < c_2}$: \,\,\,  \underline{$x^3 - 4x^2 + 3x - 1 = 0$}. \\
One has: $c_0 = 0.741, \,\, c_1 = 2.113, \,\, c_2 = -0.631$. Also: $\mu_1 = 2.215, \, \, \mu_2 = 0.451, \,\, \nu_1 = 2.861, \,\, \nu_2 = \phi = 1.333, \,\, \nu_3 = -0.195, \,\, \xi_1 = -0.431,$ and $\xi_2 = 3.097$. \\
There is only one real root: $x_1 > \xi_2$, that is $x_1 > 3.097$. This can be bounded from above by using a polynomial root bound. Both bounds given earlier yield that $x_1 < 5$. \\
The roots are: $x_1 = 3.150$ and $x_{2,3} = 0.426 \pm 0.369i$.

\item {\bf Theorem 2(II)}, $\bm{b < a^2/3, \,\, c > c_1}$: \,\,\,  \underline{$x^3 + 2x^2 + \frac{1}{2}x - 1 = 0$}.
\vskip0.01cm
One has: $c_0 = -0.259, \,\, c_1 = 0.034, \,\, c_2 = -0.552$. Also: $\mu_1 = -0.140, \, \, \mu_2 = -1.194, \,\, \nu_1 = 0.246, \,\, \nu_2 = \phi = -0.667, \,\, \nu_3 = -1.580, \,\, \xi_1 = -1.721,$ and $\xi_2 = 0.387$. \\
There is only one real root: $x_1 < \xi_1$, that is $x_1 < -1.721$. This can be bounded from bellow by using a polynomial root bound. Both bounds given earlier again agree and yield that $- 3 < x_1$. \\
The roots are: $x_1 = -2$ and $x_{2,3} = \pm 0.707i$.

\item {\bf Theorem 3(I)}, $\bm{b = a^2/3, \,\, c < a^3/27}$: \,\,\,  \underline{$x^3 - 2x^2 + \frac{4}{3}x - 2 = 0$}.
\vskip0.01cm
There is only one real root and it can be determined by completing the cube: $x^3 + ax^2 + (a^2/3)x + c = (x + a/3)^3 - a^3/27 + c$. Hence, $x_1 = -a/3 + \sqrt[3]{a^3/27 - c} = 1.862$. The other two roots are $x_{2,3} = 0.070 \pm 1.030i$.

\item {\bf Theorem 3(II)}, $\bm{b = a^2/3, \,\, c = a^3/27}$: \,\,\,  \underline{$x^3 + 5x^2 + \frac{25}{3}x + \frac{125}{27} = 0$}.
\vskip0.01cm
There is a triple real root that can be determined exactly: $x_{1,2,3} = -a/3 = -5/3$. \\

\item {\bf Theorem 3(III)}, $\bm{b = a^2/3, \,\, c > a^3/27}$: \,\,\,  \underline{$x^3 - 6x^2 + 12x + 5 = 0$}. \\
There is only one real root and it can be determined by completing the cube: $x^3 + ax^2 + (a^2/3)x + c = (x + a/3)^3 - a^3/27 + c$. Hence, $x_1 = -a/3 + \sqrt[3]{a^3/27 - c} = - 0.351$. The other two roots are $x_{2,3} = 3.176 \pm 2.036i$.

\item {\bf Theorem 4(I)}, $\bm{b > a^2/3, \,\, a \ge 0, \,\, c \le 0}$: \,\,\,  \underline{$x^3 + x^2 + 2x - 3 = 0$}. \\
There is only one real root $x_1$ and its isolation interval is $0 \le x_1 \le -c/b$, that is $0 \le x_1 \le 1.5$. \\
The roots are: $x_1 = 0.844$ and $x_{2,3} = -0.922 \pm 1.645i$.

\item {\bf Theorem 4(II)}, $\bm{b > a^2/3, \,\, a \ge 0, \,\, c > 0}$: \,\,\,  \underline{$x^3 - x^2 + 10x + 7 = 0$}. \\
There is only one real root $x_1$ and its isolation interval is min$\{-a, -c/b\} \le x_1 \le $ max$\{-a, -c/b\}$, that is $-0.7 \le x_1 < 1$. \\
The roots are: $x_1 = -0.634$ and $x_{2,3} = 0.817 \pm 0.322i$.

\item {\bf Theorem 4(III)}, $\bm{b > a^2/3, \,\, a < 0, \,\, c < 0}$: \,\,\,  \underline{$x^3 - 2x^2 + 13x - 11 = 0$}. \\
There is only one real root $x_1$ and its isolation interval is min$\{-a, -c/b\} \le x_1 \le $ max$\{-a, -c/b\}$, that is $0.846 \le x_1 \le 2$. \\
The roots are: $x_1 = 0.916$ and $x_{2,3} = 0.542 \pm 3.422i$.

\item {\bf Theorem 4(IV)}, $\bm{b > a^2/3, \,\, a < 0, \,\, c \ge 0}$: \,\,\,  \underline{$x^3 - 3x^2 + 21x + 7 = 0$}. \\
There is only one real root $x_1$ and its isolation interval is $-c/b \le x_1 \le 0$, that is $-0.333 \le x_1 \le 0$. \\
The roots are: $x_1 = -0.317$ and $x_{2,3} = 1.659 \pm 4.393i$.

\end{enumerate}

\end{document}